\documentclass[12pt]{amsart}
\usepackage{amsfonts}
\usepackage{mathrsfs}
\usepackage{amscd,amsmath,amssymb,amsfonts}
\usepackage[all]{xy}
\theoremstyle{plain}
\newtheorem{thm}{Theorem}
\newtheorem{lem}[thm]{Lemma}

\newtheorem{prop}[thm]{Proposition}
\newtheorem{conj}[thm]{Conjecture}
\theoremstyle{definition}
\newtheorem{defn}[thm]{Definition}

\newtheorem{rmk}[thm]{Remark}

\newtheorem{ex}[thm]{Example}

\numberwithin{thm}{section} \numberwithin{equation}{section}

\newcommand{\ga}[2]{\begin{gather}\label{#1}#2 \end{gather}}

\newcommand{\sE}{{\mathcal E}}
\newcommand{\sF}{{\mathcal F}}
\newcommand{\sG}{{\mathcal G}}

\newcommand{\sK}{{\mathcal K}}
\newcommand{\sL}{{\mathcal L}}

\newcommand{\sO}{{\mathcal O}}
\newcommand{\sP}{{\mathcal P}}
\newcommand{\sQ}{{\mathcal Q}}
\newcommand{\sR}{{\mathcal R}}
\newcommand{\sS}{{\mathcal S}}

\newcommand{\sU}{{\mathcal U}}
\newcommand{\sV}{{\mathcal V}}

\begin{document}
\title{elliptic curves in moduli space of stable bundles}
\author{Xiaotao Sun}
\address{Chinese Academy of Mathematics and Systems Science, Beijing, P. R. of China}
\email{xsun@math.ac.cn}
\date{September 8, 2010}
\thanks{Partially supported by NBRPC 2011CB302400 and NSFC (No. 10731030)}

\maketitle
\begin{quote}
$\qquad\,$      Dedicated to the memory of Eckart Viehweg
\end{quote}

\begin{abstract}Let $M$ be the moduli space of rank $2$ stable bundles
with fixed determinant of degree $1$ on a smooth projective curve
$C$ of genus $g\ge 2$. When $C$ is generic, we show that any
elliptic curve on $M$ has degree (respect to anti-canonical divisor
$-K_M$) at least 6, and we give a complete classification for
elliptic curves of degree $6$. Moreover, if $g>4$, we show that any
elliptic curve passing through the generic point of $M$ has degree
at least $12$. We also formulate a conjecture for higher rank.
\end{abstract}

\section{Introduction}

Let $C$ be a smooth projective curve of genus $g\ge 2$ and $\sL$ be
a line bundle of degree $d$ on $C$. Let $M:=\sS\sU_C(r,\sL)^s$ be
the moduli space of stable vector bundles on $C$ of rank $r$ and
with fixed determinant $\sL$, which is a smooth qusi-projective Fano
variety with ${\rm Pic}(M)=\Bbb Z\cdot \Theta$ and
$-K_M=2(r,d)\Theta$, where $\Theta$ is an ample divisor. Let $B$ be
a smooth projective curve of genus $b$. The degree of a curve $\phi:
B\to M$ is defined to be ${\rm deg}\phi^*(-K_M)$. It seems quite
natural to ask what is the lower bound of degree and to classify the
curves of lowest degree.

When $B=\Bbb P^1$, we have determined all $\phi:\Bbb P^1\to M$ with
lowest degree in \cite{MS} and all $\phi:\Bbb P^1\to M$ passing
through the generic point of $M$ with lowest degree in \cite{Sun}.
In fact, one can construct $\phi:\Bbb P\to M$ for various projective
spaces $\Bbb P$ such that $\phi^*(-K_M)=\sO_{\Bbb P}(2(r,d))$, and
$\phi:\Bbb P^{r-1}\to M$ passing through the generic point of $M$
such that $\phi^*(-K_M)=\sO_{\Bbb P^{r-1}}(2r)$. Then it was proved
in \cite{MS} and \cite{Sun} that the images of lines in these
projective spaces exhaust all minimal rational curves on $M$ (resp.
minimal rational curves passing through generic point of $M$). Some
applications of the results were also pointed out in \cite{MS} and
\cite{Sun}. Thus it is natural to ask what are the situation when
$b>0$. This note is a start to study the case of $b=1$. It may
happen that the normalization of $\phi(B)$ is $\Bbb P^1$. To avoid
this case, we call $\phi:B\to M$ \textbf{an essential elliptic
curve} of $M$ if the normalization of $\phi(B)$ is an elliptic
curve.

It is easy to construct essential elliptic curves of degree $6(r,d)$
on $M$, and essential elliptic curves of degree $6r$ that pass
through the generic point of $M$. For example, for smooth elliptic
curves $B\subset\Bbb P$ of degree $3$, the morphism $\phi:\Bbb P\to
M$ defines essential elliptic curves $\phi|_B:B\to M$ of degree
$6(r,d)$ (See Example \ref{ex3.6}), which are called
\textbf{elliptic curves of split type}. For smooth elliptic curves
$B\subset\Bbb P^{r-1}$ of degree $3$, the morphism $\phi:\Bbb
P^{r-1}\to M$ defines essential elliptic curves $\phi|_B:B\to M$ of
degree $6r$ passing through the generic point of $M$ (See Example
\ref{ex3.5}), which are called \textbf{elliptic curves of Hecke
type}. Are they minimal elliptic curves of $M$ (resp. minimal
elliptic curves passing through generic point of $M$)?  Do they
exhaust all minimal essential elliptic curves on $M$ (See Conjecture
\ref{conj4.8} for detail)?

In this note, we consider the case that $r=2$ and $d=1$, then $M$ is
a smooth projective fano manifold of dimension $3g-3$. When $C$ is
generic, we show that any essential elliptic curve $\phi:B\to M$ has
degree at least $6$, and it must be an \textbf{elliptic curve of
split type} if it has degree $6$ (See Theorem \ref{thm4.6}). When
$g>4$ and $C$ is generic, we show that any essential elliptic curve
$\phi:B\to M$ passing through the generic point of $M$ have degree
at least $12$ (See Theorem \ref{thm4.7}). When $C$ is generic, there
is no nontrivial morphism from $C$ to an elliptic curve, which
implies that ${\rm Pic}(C\times B)={\rm Pic}(C)\times {\rm Pic}(B)$.
It is the condition that we need through the whole paper.

We give a brief description of the article. In Section 2, we show a
formula of degree for general case. In Section 3, we show how the
general formula implies the known case $B=\Bbb P^1$ and construct
the examples of essential elliptic curves of degree $6(r,d)$ and
$6r$. In Section 4, we prove the main theorems (Theorem \ref{thm4.6}
and Theorem \ref{thm4.7}), which is the special case $r=2$, $d=1$ of
Conjecture \ref{conj4.8}. Although I believe the conjecture, I leave
the case of $r>2$ to other occasion.

\section{The degree formula of curves in moduli spaces}

Let $C$ be a smooth projective curve of genus $g\ge 2$ and $\sL$ a
line bundle on $C$ of degree $d$. Let $M=\sS\sU_C(r,\sL)^s$ be the
moduli spaces of stable bundles on $C$ of rank $r$, with fixed
determinant $\sL$. It is well-known that $Pic(M)=\Bbb Z\cdot
\Theta$, where $\Theta$ is an ample divisor.

\begin{lem}\label{lem2.1} For any smooth projective curve $B$ of
genus $b$, if $$\phi:B\to M$$ is defined by a vector bundle $E$ on
$C\times B$, then
$${\rm deg}\phi^*(-K_M)=c_2(\sE nd^0(E))=2rc_2(E)-(r-1)c_1(E)^2:=\Delta(E)$$
\end{lem}

\begin{proof}
In general, there is no universal bundle on $C\times M$, but there
exist vector bundle $\sE nd^0$ and projective bundle $\sP$ on
$C\times M$ such that $\sE nd^0|_{C\times\{[V]\}}=\sE nd^0(V)$ and
$\sP|_{C\times\{[V]\}}=\Bbb P(V)$ for any $[V]\in M$. Let
$\pi:C\times M\to M$ be the projection, then $T_M=R^1\pi_*(\sE
nd^0)$, which commutes with base changes since $\pi_*(\sE nd^0)=0$.

For any curve $\phi:B\to M$, let $X:=C\times B$, $\Bbb
E=(id\times\phi)^*\sE nd^0$ and $\pi:X=C\times B\to B$ still denote
the projection. Then $\phi^*T_M=R^1\pi_*\Bbb E$. By Riemann-Roch
theorem, we have
$${\rm deg}\phi^*(-K_M)=\chi(R^1\pi_*\Bbb E)+(r^2-1)(g-1)(b-1).$$
By using Leray spectral sequence and $\chi(\Bbb E)={\rm deg}(ch(\Bbb
E)\cdot td(T_X))_2$, we have $\chi(R^1\pi_*\Bbb E)=-\chi(\Bbb
E)=c_2(\Bbb E)-(r^2-1)(g-1)(b-1)$, hence
$${\rm deg}\phi^*(-K_M)=c_2(\Bbb E).$$
If $\phi: B\to M$ is defined by a vector bundle $E$ on $X=C\times
B$, then $\Bbb E=\sE nd^0(E)$ (cf. the proof of lemma 2.1 in
\cite{Sun}). Thus
$${\rm deg}\phi^*(-K_M)=c_2(\sE nd^0(E))=2rc_2(E)-(r-1)c_1(E)^2.$$
\end{proof}

Let $f: X\to C$ be the projection. Then, for any vector bundle $E$
on $X$, there is a relative Harder-Narasimhan filtration (cf Theorem
2.3.2, page 45 in \cite{HL})
$$0=E_0\subset E_1\subset \,\cdots\,\subset E_n=E$$
such that $F_i=E_i/E_{i-1}$ ($i=1,\,...\, n$) are flat over $C$ and
its restriction to general fiber $X_p=f^{-1}(p)$ is the
Harder-Narasimhan filtration of $E|_{X_p}$. Thus $F_i$ are
semi-stable of slop $\mu_i$ at generic fiber of $f:X\to B$ with
$\mu_1>\mu_2>\,\cdots\,>\mu_n.$ Then we have the following theorem

\begin{thm}\label{thm2.2} For any vector bundle $E$ of rank $r$ on $X$,
let $$0=E_0\subset E_1\subset \,\cdots\,\subset E_n=E$$ be the
relative Harder-Narasimhan filtration over $C$ with
$F_i=E_i/E_{i-1}$ and $\mu_i=\mu(F_i|_{f^{-1}(x)})$ for generic
$x\in C$. Let $\mu(E)$ and $\mu(E_i)$ denote the slop of
$E|_{\pi^{-1}(b)}$ and $E_i|_{\pi^{-1}(b)}$ for generic $b\in B$.
Then, if $${\rm Pic}(C\times B)={\rm Pic}(C)\times {\rm Pic}(B),$$
we have the following formula
\ga{2.1}{\Delta(E)=2r\left(\aligned&\sum_{i=1}^n\left(c_2(F_i)-\frac{{\rm
rk}(F_i)-1}{2\,{\rm rk}(F_i)}c_1(F_i)^2 \right)\\
&+\sum_{i=1}^{n-1}(\mu(E)-\mu(E_i)){\rm
rk}(E_i)(\mu_i-\mu_{i+1})\endaligned\right).}
\end{thm}

\begin{proof} It is easy to see that
$$\aligned 2c_2(E)&=2\sum_{i=1}^nc_2(F_i)+2\sum^n_{i=1}c_1(E_{i-1})c_1(F_i)\\
&=2\sum_{i=1}^nc_2(F_i)+c_1(E)^2-\sum_{i=1}^nc_1(F_i)^2.\endaligned$$
Thus
$$\Delta(E)=
2r\sum_{i=1}^nc_2(F_i)+c_1(E)^2-r\sum_{i=1}^nc_1(F_i)^2.$$

Let $r_i$ be the rank of $F_i$ and $d_i$ be the degree of $F_i$ on
the generic fiber of $\pi:C\times B\to B$. Then we can write
$$c_1(F_i)=f^*\sO_C(d_i) + \pi^*\sO_B(r_i\mu_i)$$
where $\sO_C(d_i)$ (resp. $\sO_B(r_i\mu_i)$) denotes a divisor of
degree $d_i$ (resp. degree $r_i\mu_i$) of $C$ (resp. $B$). Note that
$$c_1(F_i)^2=2d_ir_i\mu_i,\qquad c_1(E)^2=2d\sum^n_{i=1}r_i\mu_i$$
we have
$$\aligned&\Delta(E)=2r\left(\sum_{i=1}^nc_2(F_i)+\mu(E)\sum_{i=1}^nr_i\mu_i
-\sum^n_{i=1}d_ir_i\mu_i\right)\\
&=2r\left(\sum_{i=1}^n(c_2(F_i)-(r_i-1)d_i\mu_i)+\mu(E)\sum_{i=1}^nr_i\mu_i
-\sum^n_{i=1}d_i\mu_i\right).\endaligned$$ Let ${\rm deg}(E_i)$
denote the degree of $E_i$ on the generic fiber of $$\pi:C\times
B\to B.$$ Using $d_i={\rm deg}(E_i)-{\rm deg}(E_{i-1})$ and
$r_i={\rm rk}(E_i)-{\rm rk}(E_{i-1})$, we have
$$\mu(E)\sum_{i=1}^nr_i\mu_i
-\sum^n_{i=1}d_i\mu_i=\sum^{n-1}_{i=1}(\mu(E)-\mu(E_i)){\rm
rk}(E_i)(\mu_i-\mu_{i+1}).$$ Since $d_i\mu_i=c_1(F_i)^2/2r_i$, we
get the formula
$$\Delta(E)=2r\left(\aligned&\sum_{i=1}^n\left(c_2(F_i)-\frac{r_i-1}{2r_i}c_1(F_i)^2\right)\\
&+\sum_{i=1}^{n-1}(\mu(E)-\mu(E_i)){\rm
rk}(E_i)(\mu_i-\mu_{i+1})\endaligned\right).$$

\end{proof}

\begin{rmk}\label{rmk2.3} I do not know if the formula holds without
the assumption that ${\rm Pic}(C\times B)={\rm Pic}(C)\times {\rm
Pic}(B)$. On the other hand, the assumption holds when $B$ is an
elliptic curve and $C$ is generic.
\end{rmk}

\begin{thm}\label{thm2.4} For any torsion free sheaf $\sF$ on $X=C\times
B$, if its restriction to a fiber of $f: X=C\times B\to C$ is
semi-stable, then
$$\Delta(\sF)=2\,{\rm rk}(\sF)\,c_2(\sF)-({\rm
rk}(\sF)-1)c_1(\sF)^2\ge 0.$$ If the determinants $\{{\rm
det}(\sF^{**})_x\}_{x\in C}$ are isomorphic each other, then
$\Delta(\sF)=0$ if and only if $\sF$ is locally free and satisfies
\begin{itemize}
\item All the bundles $\{\sF_x:=\sF|_{\{x\}\times B}\}_{x\in C}$ are semi-stable and
$s$-equivalent each other.
\item All the bundles $\{\sF_y:=\sF|_{C\times\{y\}}\}_{y\in B}$ are isomorphic each
other.
\end{itemize}
\end{thm}

\begin{proof} Since $\Delta(\sF)\ge \Delta(\sF^{**})$, we can assume
that $\sF$ is a vector bundle.  There is a $x\in C$ such that
$\sF_x=\sF|_{\{x\}\times B}$ is semi-stable, so is $\sE
nd^0(\sF)_x=\sE nd^0(\sF_x)$. Thus, by a theorem of Faltings (cf.
Theorem I.2. of \cite{Fa}), there is a vector bundle $V$ on $B$ such
that
$${\rm H}^0(\sE nd^0(\sF)_x\otimes V)={\rm H}^1(\sE nd^0(\sF)_x\otimes
V)=0,$$ which defines a global section $\vartheta(V)$ of the line
bundle
$$\Theta(\sE nd^0(\sF)\otimes\pi^*V)=({\rm det}f_!(\sE
nd^0(\sF)\otimes\pi^*V))^{-1}$$ such that $\vartheta(V)(x)\neq 0$.
By Grothendieck-Riemann-Roch theorem,
$$\aligned c_1({\rm det}f_!(\sE
nd^0(\sF)\otimes\pi^*V))&=f_*({\rm ch}(\sE
nd^0(\sF)\otimes\pi^*V){\rm td}(\pi^*T_B))_2\\&=-c_2(\sE
nd^0(\sF)\otimes\pi^*V)\endaligned$$ which means that the line
bundle $\Theta(\sE nd^0(\sF)\otimes\pi^*V)$ has degree $$c_2(\sE
nd^0(\sF)\otimes\pi^*V)={\rm rk}(V)\cdot c_2(\sE nd^0(\sF))={\rm
rk}(V)\cdot\Delta(\sF)$$ with a nonzero global section
$\vartheta(V)$. Thus $\Delta(\sF)\ge 0.$

If $\Delta(\sF)=0$, then $\sF=\sF^{**}$ must be locally free and
$\vartheta(V)(x)\neq 0$ for any $x\in C$, which means that for any
$x\in C$, we have
$${\rm H}^0(\sE nd^0(\sF)_x\otimes V)={\rm H}^1(\sE nd^0(\sF)_x\otimes
V)=0.$$ Then, by the theorem of Faltings, the bundles
$$\{\,\sE nd^0(\sF)_x\,\}_{x\in C}$$
are all semi-stable. Thus, for any $x\in C$, the bundle
$\sF_x:=\sF|_{\{x\}\times B}$ is semi-stable. The bundle $\sF$
defines a morphism $\phi_{\sF}: C\to \sU_B$ from $C$ to the moduli
space $\sU_B$ of semi-stable bundles on $B$, the line bundle
$\Theta(\sE nd^0(\sF)\otimes\pi^*V)$ clearly descends to a line
bundle on $\sU_B$. If the determinants ${\rm det}(\sF_x)$ ($x\in C)$
are fixed, then $${\rm deg}(\Theta(\sE nd^0(\sF)\otimes\pi^*V))=0$$
means that all $\{\sF_x\}_{x\in C}$ are $s$-equivalence.

By using a technique of \cite{He} (see Step 5 in the proof of
Theorem 4.2 in \cite{He}, see also the proof of Theorem I.4 in
\cite{Fa}), we will show
$$\sF|_{C\times\{y_1\}}\cong \sF|_{C\times\{y_2\}},\quad
\forall\,\,y_1,\,y_2\,\in B.$$ Choose a nontrivial extension $0\to
V\to V' \xrightarrow{q_1} \sO_{y_1}\to 0$ on $B$, let $\mathfrak{Q}$
be the Quot-scheme of rank $0$ and degree $1$ quotients of $V'$, and
$$0\to \mathcal{K}\to p_B^*V'\to \mathfrak{T}\to 0$$
be the tautological exact sequence on $B\times\mathfrak{Q}$. Fix a
point $x_1\in C$, then the set $q\in \mathfrak{Q}$ such that ${\rm
H}^0(\sF_{x_1}\otimes\mathcal{K}_q)={\rm
H}^1(\sF_{x_1}\otimes\mathcal{K}_q)=0$ is an open set $U\subset
\mathfrak{Q}$ and $U\neq \emptyset$ since $q_1=(0\to V\to
V'\xrightarrow{q_1} \sO_{y_1}\to 0)\in U$.

Let $\Gamma\subset B\times \mathbb{P}(V')$ be the graph of
$\mathbb{P}(V')\xrightarrow{p} B$, then
$$p_B^*V'\to p_B^*V'|_{\Gamma}=p^*V'\to \sO(1)\to 0$$
induces a quotient $p_B^*V'\to \,_{\Gamma}\sO(1)\to 0$ on $B\times
\mathbb{P}(V')$, which defines a morphism $\mathbb{P}(V')\to
\mathfrak{Q}$. It is easy to see that $\mathbb{P}(V')\to
\mathfrak{Q}$ is surjective (in fact, it is a isomorphism). Thus
there is an open $B_1\subset B$ with $y_1\in B_1$ such that for any
$y\in B_1$ there exists an exact sequence \ga{2.2}{0\to \sK_q\to
V'\xrightarrow{q} \sO_y\to 0} such that ${\rm
H}^0(\sF_{x_1}\otimes\mathcal{K}_q)={\rm
H}^1(\sF_{x_1}\otimes\mathcal{K}_q)=0,$ which implies
$${\rm H}^0(\sF_x\otimes\mathcal{K}_q)={\rm
H}^1(\sF_x\otimes\mathcal{K}_q)=0\qquad \forall\,\, x\in C$$ since
$\sF_x$ is $s$-equivalent to $\sF_{x_1}$ for any $x\in C$. Pull back
the exact sequence \eqref{2.2} by $\pi: C\times B\to B$ and tensor
with $\sF$, we have the exact sequence \ga{2.3}{0\to
\sF\otimes\pi^*\sK_q\to \sF\otimes\pi^*V'\to \sF_y\to 0.} Take
direct images of \eqref{2.2} under $f:C\times B\to C$, we have
$$\sF_y\cong f_*(\sF\otimes\pi^*V')\,,\quad \forall\,\,y\in B_1$$
which implies that all $\{\sF_y\}_{y\in B}$ are isomorphic each
other.
\end{proof}

We will need the following lemma in the later computation, whose
proof are straightforward computations (see \cite{F} for the case of
rank $1$).

\begin{lem}\label{lem2.5} Let $X$ be a smooth projective surface and
$j:D\hookrightarrow X$ be an effective divisor. Then, for any vector
bundle $V$ on $D$, we have
$$\aligned&c_1(j_*V)={\rm rk}(V)\cdot D\\&c_2(j_*V)=\frac{{\rm
rk}(V)({\rm rk}(V)+1)}{2}D^2-j_*c_1(V).\endaligned$$

\end{lem}

Recall that $X_t=f^{-1}(t)$ denotes the fiber of $f:X\to C$ and for
any vector bundle $\sF$ on $X$, $\sF_t$ denote the restrictions of
$\sF$ to $X_t$.

\begin{lem}\label{lem2.6}
Let $\sF_t\to W\to 0$ be a locally free quotient and
$$0\to\sF'\to \sF\to \, _{X_t}W\to 0$$
be the elementary transformation of $\sF$ along $W$ at $X_t\subset
X$. Then
$$\Delta(\sF)=\Delta(\sF')+2r(\mu(\sF_t)-\mu(W)){\rm rk}(W).$$
\end{lem}

\section{Mminimal rational curves and examples of elliptic curves on moduli spaces}

When $B=\Bbb P^1$, the condition ${\rm Pic}(C\times B)={\rm
Pic}(C)\times{\rm Pic}(B)$ always hold and any morphism $B\to M$ is
defined by a vector bundle on $C\times B$ (cf. Lemma 2.1 of
\cite{Sun}).

Recall that given two nonnegative integers $k$, $\ell$, a vector
bundle $W$ of rank $r$ and degree $d$ on $C$ is $(k,\ell)$-stable,
if, for each proper subbundle $W'$ of $W$, we have
$$\frac{{\rm deg}(W')+k}{{\rm rk}(W')}<\frac{{\rm deg}(W)+k-\ell}{r}.$$
The usual stability is equivalent to $(0,0)$-stability. The
$(k,\ell)$-stability is an open condition. The proofs of following
lemmas are easy and elementary (cf. \cite{NR}).

\begin{lem}\label{lem3.1} If $g\ge 3$, $M$ contains $(0,1)$-stable and
$(0,1)$-stable bundles. $M$ contains a $(1,1)$-stable bundle $W$
except $g=3$, $d$, $r$ both even.
\end{lem}

\begin{lem}\label{lem3.2} Let $0\to V\to W\to \sO_p\to 0$ be an exact
sequence, where $\sO_p$ is the $1$-dimensional skyscraper sheaf at
$p\in C$. If $\,\,W$ is $(k,\ell)$-stable, then $V$ is
$(k,\ell-1)$-stable.
\end{lem}

A curve $B\to M$ defined by $E$ on $C\times B$ passes through the
generic point of $M$ implies that $E_y:=E|_{C\times\{y\}}$ is
$(1,1)$-stable for generic $y\in B$. Thus in the formula \eqref{2.1}
of Theorem \ref{thm2.2} we have \ga{3.1}{(\mu(E)-\mu(E_i)){\rm
rk}(E_i)>1.} On the other hand, any semi-stable bundle on $B=\Bbb
P^1$ must have integer slop. By the formula \eqref{2.1} in Theorem
\ref{thm2.2}, we have $$\Delta(E)>2r$$ if $E$ is not semi-stable on
the generic fiber of $f: X=C\times \Bbb P^1\to C$.

When $E$ is semi-stable on the generic fiber of $f: X\to C$, by
tensor $E$ with a line bundle, we can assume that $E$ is trivial on
the generic fiber of $f: X\to C$. Thus $\Delta(E)=2rc_2(E)\ge 2r$
and there must be a fiber $X_t=f^{-1}(t)$ such that $E_t=E|_{X_t}$
is not semi-stable by Theorem \ref{thm2.4}. If $\Delta(E)=2r$, by
Lemma \ref{lem2.6}, we must have ${\rm rk}(W)=1$, $\mu(W)=-1$ and
$\Delta(\sF')=0$ in Lemma \ref{lem2.6}. Thus $\Delta(E)=2r$ if and
only if $E$ satisfies
$$0\to f^*V\to E\to \,_{X_t}\sO_{\Bbb P^1}(-1)\to 0$$
which defines a so called Hecke curve. Therefore we get the main
theorem in \cite{Sun}.

\begin{thm}\label{thm3.3} If $g\ge 3$, then any rational curve of $M$
passing through the generic point of $M$ has at least degree $2r$
with respect to $-K_M$.  It has degree $2r$ if and only if it is a
Hecke curve except $g=3$, $r=2$ and $(2,d)=2$.
\end{thm}

At the end of this section, we give some examples of elliptic curves
on $M$. Let us recall the construction of Hecke curves. Let
$\sU_C(r,d-1)$ be the moduli space of stable bundles of rank $r$ and
degree $d-1$. Let
$$\mathfrak{O}\subset\sU_C(r,d-1)$$ be the open set of $(1,0)$-stable
bundles. Let $C\times\mathfrak{O}\xrightarrow\psi J^d(C)$ be defined
as $\psi(x, V)=\sO_C(x)\otimes{\rm det}(V)$ and
$$\mathscr{R}_C:=\psi^{-1}(\sL)\subset C\times \mathfrak{O},$$
which consists of the points $(x,V)$ such that $V$ are
$(1,0)$-stable bundles on $C$ with ${\rm det}(V)=\sL(-x)$. There
exists a projective bundle
$$p:\mathscr{P}\to \mathscr{R}_C$$
such that for any $(x,V)\in \mathscr{R}_C$ we have $p^{-1}(x,V)=\Bbb
P(V_x^{\vee})$. Let
$$V_x^{\vee}\otimes\sO_{\mathbb{P}(V_x^{\vee})}\to\sO_{\Bbb P
(V_x^{\vee})}(1)\to 0$$ be the universal quotient, $f:C\times \Bbb
P(V_x^{\vee})\to C$ be the projection, and
$$0\to\mathscr{E}^{\vee}\to
f^*V^{\vee}\to\,_{\{x\}\times\Bbb P(V_x^{\vee})} \sO_{\Bbb
P(V_x^{\vee})}(1)\to 0$$ where $\mathscr{E}^{\vee}$ is defined to
the kernel of the surjection. Take dual, we have \ga{3.2} {0\to
f^*V\to\mathscr{E}\to\,_{\{x\}\times\Bbb P(V_x^{\vee})} \sO_{\Bbb
P(V_x^{\vee})}(-1)\to 0,} which, at any point $\xi=(V_x^{\vee}\to
\Lambda\to 0)\in \Bbb P(V_x^{\vee})$, gives exact sequence
$$0\to V\xrightarrow\iota\mathscr{E}_{\xi}\to \,\sO_x\to 0$$ on $C$
such that ${\rm ker}(\iota_x)=\Lambda^{\vee}\subset V_x$. $V$ being
$(1,0)$-stable implies stability of $\mathscr{E}_{\xi}$. Thus
\eqref{3.2} defines \ga{3.3}{\Psi_{(x,V)}:\,\,\Bbb
P(V_x^{\vee})=p^{-1}(x,V)\to M.}

\begin{defn}\label{defn3.4} The images (under $\{\Psi_{(x,V)}\}_{(x,V)\in\mathscr{R}_C}$) of lines in the fibres of
$p: \mathscr{P}\to \mathscr{R}_C$ are the so called \textbf{Hecke
curves} in $M$. The images (under
$\{\Psi_{(x,V)}\}_{(x,V)\in\mathscr{R}_C}$) of elliptic curves in
the fibres of
$$p: \mathscr{P}\to \mathscr{R}_C$$ are called \textbf{elliptic curves of
Hecke type}.
\end{defn}

It is known (cf. \cite[Lemma 5.9]{NR}) that the morphisms in
\eqref{3.3} are closed immersions. By a straightforward computation,
we have \ga{3.4}{\Psi_{(x,V)}^*(-K_M)=\sO_{\Bbb P(V_x^{\vee})}(2r).}
For any point $[W]\in M$ and $(W_x\to \Bbb C\to 0)\in \Bbb P(W_x)$,
where $W$ is $(1,1)$-stable, we define a $(1,0)$-stable bundle $V$
by
$$0\to V\xrightarrow\alpha  W\to\,_x\Bbb C\to 0.$$
Then the images of $p^{-1}(x,V)=\Bbb P(V_x^{\vee})$ are projective
spaces that pass through $[W]\in M$, and the images of lines $\ell
\subset \Bbb P(V_x^{\vee})$ that pass through $[{\rm
ker}(\alpha_x)]\in \Bbb P(V_x^{\vee})$ are Hecke curves passing
through $[W]\in M$.

\begin{ex}\label{ex3.5} When $g\ge 4$ and $r>2$, for generic $[W]\in M$,
the images of smooth elliptic curves $B\subset \Bbb
P(V_x^{\vee})$ with degree $3$ and $[{\rm ker}(\alpha_x)]\in B$ are
smooth elliptic curves on $M$ that pass through $[W]\in M$, which
have degree $6r$ by \eqref{3.4}.
\end{ex}

If we do not require the curve $\phi:B\to M$ passing through generic
point of $M$, we may construct rational curves and elliptic curves
with smaller degree. Let us recall the Construction 2.3 from
\cite{MS}.

For any given $r$ and $d$, let $r_1$, $r_2$ be positive integers and
$d_1$, $d_2$ be integers that satisfy the equalities $r_1+r_2=r$,
$d_1+d_2=d$ and
$$r_1\frac{d}{(r,d)}-d_1\frac{r}{(r,d)}=1,\quad d_2\frac{r}{(r,d)}-r_2\frac{d}{(r,d)}=1.$$
Let $\sU_C(r_1,d_1)$ (resp. $\sU_C(r_2,d_2)$) be the moduli space of
stable vector bundles with rank $r_1$ (resp. $r_2$) and degree $d_1$
(resp. $d_2$). Then, since $(r_1,d_1)=1$ and $(r_2,d_2)=1$, there
are universal vector bundles $\sV_1$, $\sV_2$ on $C\times
\sU_C(r_1,d_1)$ and $C\times\sU_C(r_2,d_2)$ respectively. Consider
$$\sU_C(r_1,d_1)\times\sU_C(r_2,d_2)\xrightarrow{{\rm det}(\bullet)\times
{\rm det}(\bullet)} J^{d_1}_C\times
J^{d_2}_C\xrightarrow{(\bullet)\otimes (\bullet)}J^d_C,$$ let
$\sR(r_1,d_1)$ be its fiber at $[\sL]\in J^d_C$. The pullback of
$\sV_1$, $\sV_2$ by the projection $C\times\sR(r_1,d_1)\to
C\times\sU_C(r_i,d_i)$ ($i=1,2$) is still denoted by $\sV_1$,
$\sV_2$ respectively. Let $p:C\times\sR(r_1,d_1)\to \sR(r_1,d_1)$
and $$\sG=R^1p_*(\sV_2^{\vee}\otimes\sV_1),$$ which is locally free
of rank $r_1r_2(g-1)+(r,d)$. Let $$q: P(r_1,d_1)=\Bbb
P(\sG)\to\sR(r_1,d_1)$$ be the projective bundle parametrzing
$1$-dimensional subspaces of $\sG_t$ ($t\in\sR(r_1,d_1)$) and
$f:C\times P(r_1,d_1)\to C$, $\pi:C\times P(r_1,d_1)\to P(r_1,d_1)$
be the projections. Then there is a universal extension
\ga{3.5}{0\to (id\times
q)^*\sV_1\otimes\pi^*\sO_{P(r_1,d_1)}(1)\to\sE\to (id\times
q)^*\sV_2 \to 0} on $C\times P(r_1,d_1)$ such that for any
$x=([V_1], [V_2], [e])\in P(r_1,d_1)$, where $[V_i]\in
\sU_C(r_i,d_i)$ with ${\rm det}(V_1)\otimes {\rm det}(V_2)=\sL$ and
$[e]\subset {\rm H}^1(C,V_2^{\vee}\otimes V_1)$ being a line through
the origin, the bundle $\sE|_{C\times\{x\}}$ is the isomorphic class
of vector bundles $E$ given by extensions
$$0\to V_1\to V\to V_2\to 0 $$
that defined by vectors on the line $[e]\subset {\rm
H}^1(C,V_2^{\vee}\otimes V_1)$. Then $V$ must be stable by
\cite[Lemma 2.2]{MS}, and the sequence \eqref{3.5} defines
$$\Phi:P(r_1,d_1)\to \sS\sU_C(r,\sL)^s=M.$$
On each fiber $q^{-1}(\xi)=\Bbb P({\rm H}^1(V_2^{\vee}\otimes V_1))$
at $\xi=(V_1,V_2)$, the morphisms
\ga{3.6}{\Phi_{\xi}:=\Phi|_{q^{-1}(\xi)}:q^{-1}(\xi)=\Bbb P({\rm
H}^1(V_2^{\vee}\otimes V_1))\to M} is birational and
$\Phi_{\xi}^*(-K_M)=\sO_{\Bbb P({\rm H}^1(V_2^{\vee}\otimes
V_1))}(2(r,d))$ by \cite[Lemma 2.4]{MS}.

\begin{ex}\label{ex3.6} The images of lines $\ell\subset \Bbb P({\rm H}^1(V_2^{\vee}\otimes
V_1))$ are rational curves of degree $2(r,d)$ on $M$, which is
clearly the minimal degree since $-K_M=2(r,d)\Theta$. For smooth
elliptic curves $B\subset \Bbb P({\rm H}^1(V_2^{\vee}\otimes V_1))$
of degree $3$, the images of $\Phi_{\xi}: B\to M$ are of degree
$6(r,d)$. For any smooth elliptic curve $B\subset q^{-1}(\xi)$
($\forall\,\,\xi\in \sR(r_1,d_1)$), the images of $\Phi_{\xi}:B\to
M$ are called \textbf{elliptic curves of split type}.
\end{ex}

\section{Minimal elliptic curves on moduli spaces}

In this section, we consider the moduli space $M$ of rank $2$ stable
bundles on $C$ with a fixed determinant $\sL$ of degree $1$. We also
assume that the curve $C$ is generic in the sense that $C$ admits no
surjective morphism to an elliptic curve. With this assumption, we
know that ${\rm Pic}(C\times B)={\rm Pic}(C)\times {\rm Pic}(B)$ for
any elliptic curve $B$.

For a morphism $\phi: B\to M$, it may happen that the normalization
of $\phi(B)$ is a rational curve. To avoid this case, we make the
following definition

\begin{defn}\label{defn4.1} $\phi:B\to M$ is called an essential
elliptic curve of $M$ if the normalization of $\phi(B)$ is an
elliptic curve.
\end{defn}

For any morphism $\phi: B\to M$, let $E$ be the vector bundle on
$X=C\times B$ that defines $\phi$. It will be free to tensor $E$
with a pull-back of line bundles on $B$. In this section, $B$ will
always denote an elliptic curve.

\begin{prop}\label{prop4.2} Let $\phi: B\to M$ be an essential
elliptic curve of $M$ defined by a vector bundle $E$. If $E$ is not
semi-stable on the generic fiber of $f:X\to C$, then
$$\Delta(E)\ge 6.$$
If $g=g(C)\ge 4$ and the curve $\phi: B\to M$ passes through the
generic point of $M$, then
$$\Delta(E)> 12.$$
\end{prop}

\begin{proof} Let $0\to E_1\to E\to F_2\to 0$ be the relative
Harder-Narasimhan filtration over $C$. Then we have exact sequence
$$0\to E_1|_{X_t}\to E|_{X_t}\to F_2|_{X_t}\to 0$$
on each fiber $X_t=\{t\}\times B$ of $f:X\to C$ since $E_1$, $F_2$
are flat over $C$. Thus $E_1$ is locally free (cf. Lemma 1.27 of
\cite{Si}) and
\ga{4.1}{\Delta(E)=4c_2(F_2)+4(\mu(E)-\mu(E_1))(\mu_1-\mu_2)} where
$\mu_1={\rm deg}(E_1|_{X_t})$, $\mu_2={\rm deg}(F_2|_{X_t})$ for
$t\in C$ (cf. Theorem 2.2).

That $0\to E_1\to E\to F_2\to 0$ is the relative Harder-Narasimhan
filtration over $C$ means for almost $t\in C$ the exact sequences
$$0\to E_1|_{X_t}\to E|_{X_t}\to F_2|_{X_t}\to 0$$
are the Harder-Narasimhan filtration of $E|_{X_t}$, which in
particular means that $F_2$ is locally free over $f^{-1}(C\setminus
T)$ where $T\subset C$ is a finite set. Thus \ga{4.2}{0\to
E_1|_{C\times\{y\}}\to E|_{C\times\{y\}}\to F_2|_{C\times\{y\}}\to
0\,\,,\quad \forall\,\,y\in B} are exact sequences, which imply that
$F_2$ is also $B$-flat.

If $c_2(F_2)=0$, then $F_2$ is a line bundle and there are line
bundles $V_1$, $V_2$ on $C$ such that
$$E_1=f^*V_1\otimes\pi^*\sO(\mu_1),\quad
F_2=f^*V_2\otimes\pi^*\sO(\mu_2)$$ where $\sO(\mu_i)$ denote line
bundles on $B$ of degree $\mu_i$. Replace $E$ by $E\otimes
\pi^*\sO(-\mu_2)$, we can assume that $E$ satisfies \ga{4.3}{0\to
f^*V_1\otimes\pi^*\sO(\mu_1-\mu_2)\to E\to f^*V_2\to 0.} Let
$d_i={\rm deg}(V_i)$ and $J=\{(L_1,L_2)\in J_C^{d_1}\times
J_C^{d_2}\,|\,L_1\otimes L_2=\sL\,\}.$ Then there is a projective
bundle $q:P\to J$ and an universal extension \ga{4.4}{0\to (id\times
q)^*\sV_1\otimes\pi^*\sO_P(1)\to\sE\to (id\times q)^*\sV_2 \to 0} on
$C\times P$ such that for any $x=([V_1], [V_2], [e])\in P$, where
$[V_i]\in J_C^{d_i}$ with $V_1)\otimes V_2=\sL$ and $[e]\subset {\rm
H}^1(C,V_2^{-1}\otimes V_1)$ being a line through the origin, the
bundle $\sE|_{C\times\{x\}}$ is the isomorphic class of vector
bundles $V$ given by extensions $0\to V_1\to V\to V_2\to 0 $ that
defined by vectors on the line $[e]\subset {\rm
H}^1(C,V_2^{-1}\otimes V_1)$, where $\sV_i$ denote the pullback
(under $C\times J\to C\times J_C^{d_i}$) of universal line bundles,
and $\pi:C\times P\to P$ denote the projection. Thus the exact
sequence \eqref{4.3} induces a morphism \ga{4.5}{\psi:B\to \Bbb
P^{d_2-d_1+g-2}= q^{-1}(V_1,V_2)\subset P} such that
$\sO(\mu_1-\mu_2)=\psi^*\sO_P(1)$ and $\phi:B\to M$ factors through
$\psi:B\to \psi(B)\subset\Bbb P^{d_2-d_1+g-2}$, which implies that
the normalization of $\psi(B)$ is an elliptic curve. Hence
$\mu_1-\mu_2\ge 3$ and $\Delta(E)\ge 6$ by \eqref{4.1}. If
$\phi:B\to M$ passes through the generic point, then
$\mu(E)-\mu(E_1)>1$ and $\Delta(E)>12$.

If $c_2(F_2)\neq 0$, $F_2$ is not locally free, which implies that
there is a $y_0\in B$ such that $F_2|_{C\times\{y_0\}}$ has torsion
$\tau(F_2|_{C\times\{y_0\}})\neq 0$ since $F_2$ is $B$-flat (cf.
Lemma 1.27 of \cite{Si}). Let
\ga{4.6}{0\to\tau(F_2|_{C\times\{y_0\}})\to F_2|_{C\times\{y_0\}}\to
F_2^0\to 0.}  Then $F_2^0$ being a quotient line bundle of
$E|_{C\times\{y_0\}}$ implies
$${\rm deg}(F_2^0)>\mu(E|_{C\times\{y_0\}})=\frac{1}{2}$$
since $E|_{C\times\{y_0\}}$ is stable. By sequences \eqref{4.2} and
\eqref{4.6}, we have
$$\mu(E_1)={\rm deg}(E_1|_{C\times\{y_0\}})=1-{\rm deg}(F_2^0)-{\rm
dim}\,\tau(F_2|_{C\times\{y_0\}})\le -1$$ which, by the formula
\eqref{4.1}, implies that $$\Delta(E)\ge 4
c_2(F_2)+4(\frac{1}{2}+1)(\mu_1-\mu_2)\ge 10.$$

When $\phi:B\to M$ passes through a generic point, in order to show
$\Delta(E)>12$, we note that $c_2(F_2)\neq 0$ and $F_2$ being
$C$-flat also imply that there exists a $t_0\in C$ such that
$F_2|_{X_{t_0}}$ has torsion $\tau(F_2|_{X_{t_0}})\neq 0$. Let
$0\to\tau(F_2|_{X_{t_0}})\to F_2|_{X_{t_0}}\to \sQ\to 0$ and
$E'={\rm ker}(E\to\,_{X_{t_0}}\sQ)$, then $$0\to E'\to
E\to\,_{X_{t_0}}\sQ\to 0$$ which, for any $y\in B$, induces exact
sequence \ga{4.7}{0\to E'|_{C\times\{y\}}\to
E|_{C\times\{y\}}\to\,_{(t_0,y)}\sQ\to 0.} Thus all
$E'_y:=E'|_{C\times\{y\}}$ are semi-stable of degree $0$. If
$\phi:B\to M$ passes through a generic point, then there is a
$y_0\in B$ such that $E_{y_0}$ is $(1,1)$-stable on
$X_{y_0}=C\times\{y_0\}$, thus $E'_{y_0}$ is stable by \eqref{4.7}
and Lemma \ref{lem3.2}. This implies that $\Delta(E')>0$. Otherwise
$\{E'_y\}_{y\in B}$ are $s$-equivalent by applying Theorem
\ref{thm2.4} to $\pi:X\to B$, which implies $E'=f^*V\otimes \pi^*L$
for a stable bundle $V$ on $C$ and a line bundle $L$ on $B$. Then
$E_t=E'_t=L\oplus L$ for any $t\neq t_0$, which is a contradiction
since $E$ is not semi-stable on the generic fiber of $f:X\to C$.

To compute $\Delta(E')$, consider the Harder-Narasimhan filtration
$$0\to E'_1\to E'\to F'_2\to 0$$
over $C$, let $\mu_1'={\rm deg}(E'_1|_{X_t})$, $\mu_2'={\rm
deg}(F'_2|_{X_t})$ for $t\in C$, then
$$\Delta(E')=4c_2(F'_2)+4(\mu(E')-\mu(E'_1))(\mu_1'-\mu_2')\ge 8.$$
To see it, we can assume $c_2(F'_2)=0$, then there are line bundles
$V_i'$ on $C$ and line bundles $\sO(\mu_i')$ on $B$ of degree
$\mu_i'$ such that
$$0\to f^*V'_1\otimes\pi^*\sO(\mu_1'-\mu_2')\to
E'\otimes\pi^*\sO(-\mu_2')\to f^*V'_2\to 0$$ which defines a
morphism $\psi:B\to \Bbb P$ to a projective space such that
$\sO(\mu_1'-\mu_2')=\psi^*\sO_{\Bbb P}(1)$. Thus $\mu_1'-\mu_2'\ge
2$ and $\Delta(E')\ge 8$. Then
$$\Delta(E)=\Delta(E')+4(\mu(E|_{X_{t_0}})-\mu(\sQ))\ge \Delta(E')+6\ge
14.$$

\end{proof}

Now we consider the case that $E$ is semi-stable on the generic
fiber of $f: X\to C$. We can assume $0\le {\rm deg}(E|_{X_t})\le 1$
on $X_t=f^{-1}(t)$.

\begin{prop}\label{prop4.3} When $E$ is semi-stable of degree $1$ on the generic
fiber of $f:X\to C$, we have $\Delta(E)\ge 10$. If $g>4$ and
$\phi:B\to M$ passes through the generic point, then $\Delta(E)\ge
14.$
\end{prop}

\begin{proof} It is easy to see that there is a unique stable rank $2$ vector
bundle with a fixed determinant of degree $1$ on an elliptic curve.
Thus $\Delta(E)>0$ if and only if there exists $t_1\in C$ such that
$E_{t_1}=E|_{X_{t_1}}$ is not semi-stable.

Let $E_{t_1}\to \sO(\mu_1)\to 0$ be the quotient of minimal degree
and
$$0\to E^{(1)}\to E\to \,_{X_{t_1}}\sO(\mu_1)\to 0$$
be the elementary transformation of $E$ along $\sO(\mu_1)$ at
$X_{t_1}$. If $E^{(i)}$ is defined and $\Delta(E^{(i)})>0$, let
$t_{i+1}\in C$ such that $E^{(i)}_{t_{i+1}}=E^{(i)}|_{X_{t_{i+1}}}$
is not semi-stable and $E^{(i)}_{t_{i+1}}\to\sO(\mu_{i+1})\to 0$ be
the quotient of minimal degree, then we define $E^{(i+1)}$ to be the
elementary transformation of $E^{(i)}$ along $\sO(\mu_{i+1})$ at
$X_{t_{i+1}}$, namely $E^{(i+1)}$ satisfies the exact sequence
\ga{4.8}{0\to E^{(i+1)}\to E^{(i)}\to
\,_{X_{t_{i+1}}}\sO(\mu_{i+1})\to 0.} Let $s$ be the minimal integer
such that $\Delta(E^{(s)})=0$. Then \ga{4.9}{\Delta(E)=2\cdot
s-4\sum^s_{i=1}\mu_i} where $\mu_i\le 0$ ($i=1\,,\,2\,,...\,,s$).
Take direct image of \eqref{4.8}, we have \ga{4.10}{0\to
f_*E^{(s)}\to f_*E^{(s-1)}\to\,_{t_{s}}{\rm H}^0(\sO(\mu_{s}))\to 0}
(since $R^1f_*E^{(s)}=0$) and ${\rm deg}(f_*E^{(i+1)})\le {\rm
deg}(f_*E^{(i)})$, which imply \ga{4.11} {{\rm deg}(f_*E^{(s)})\le
{\rm deg}(f_*E)-{\rm dim}\,{\rm H}^0(\sO(\mu_s)).} Restrict
\eqref{4.8} to a fiber $X_y=\pi^{-1}(y)$,  we have exact sequence
$$0\to E^{(i+1)}_y\to E^{(i)}_y\to\, _{(t_{i+1}, y)}\Bbb C\to 0,$$
which implies that \ga{4.12}{{\rm deg}(E^{(s)}_y)={\rm
deg}(E_y)-s=1-s.} On the other hand, by Theorem \ref{thm2.4},
$\Delta(E^{(s)})=0$ implies that there exist a stable rank $2$
vector bundle $V$ of degree $1$ on $B$ and a line bundle $L$ on $C$
such that $ E^{(s)}=\pi^*V\otimes f^*L$. It is easy to see $${\rm
deg}(E^{(s)}_y)=2\,{\rm deg}(L)=2\,{\rm deg}(f_*E^{(s)}).$$ Thus,
combine \eqref{4.11} and \eqref{4.12}, we have the inequality
\ga{4.13}{s\ge 1-2\,{\rm deg}(f_*E)+2\,{\rm dim}\,{\rm
H}^0(\sO(\mu_s)).}

We claim that ${\rm deg}(f_*E)\le -1$. To show it, consider
\ga{4.14}{0\to \sF':=f^*(f_*E)\to E\to \sF\to 0} where $\sF$ is
locally free on $f^{-1}(C\setminus T)$ and $T\subset C$ is a finite
set such that $E_t$ ($t\in T$) is not semi-stable. Thus, for any
$y\in B$, the sequence \ga{4.15}{0\to \sF'_y\to E_y\to \sF_y\to 0}
is still exact, which implies that $\sF$ is $B$-flat (cf. Lemma
2.1.4 of \cite{HL}). The sequence \eqref{4.15} already implies ${\rm
deg}(f_*E)={\rm deg}(\sF'_y)\le 0$ since $E_y$ is stable of degree
$1$. Thus $\sF$ can not be locally free since
$$4\cdot c_2(\sF)=\Delta(E)-4\cdot{\rm deg}(f_*E)+2>0.$$
Then there is at least a $y_0\in B$ such that $\sF_{y_0}$ has
torsion, otherwise $\sF$ is locally free (cf. Lemma 1.27 of
\cite{Si}). The stability of $E_{y_0}$ implies that $\sF_{y_0}/
torsion$ has degree at least $1$. Thus ${\rm deg}(\sF_{y_0})\ge 2$
and
$${\rm deg}(f_*E)={\rm deg}(\sF_{y_0}')\le -1,$$
which means $s\ge 3+2 \,{\rm dim}\,{\rm H}^0(\sO(\mu_s))$.
Therefore, if $\mu_s<0$,  we have $\Delta(E)\ge 2\cdot s+4\ge 10$ by
\eqref{4.9}. If $\mu_s=0$, by tensoring $E$ with
$\pi^*\sO(\mu_s)^{-1}$, we may assume ${\rm dim}\,{\rm
H}^0(\sO(\mu_s))=1$, then $s\ge 5$ and $$ \Delta(E)\ge 10.$$

If $\phi: B\to M$ passes through the generic point of $M$, we claim
that ${\rm deg}(f_*E)\le -2$, which implies $\Delta(E)\ge 14$. To
prove the claim, assume ${\rm deg}(f_*E)=-1$, we will show that
$\phi(B)$ lies in a given divisor. Note that $\sF_y$ must be locally
free of degree $2$ for generic $y\in B$ (if $\sF_y$ has nontrivial
torsion, then $E_y$ has a quotient line bundle of degree at most
$1$, which is impossible since $E_y$ is $(1,1)$-stable for generic
$y\in B$). Thus $E_y$ satisfies $0\to \xi\to E_y\to \xi^{-1}\otimes
\sL\to 0$ where $\xi$ is a line bundle of degree $-1$ on $C$. The
locus of such bundles has dimension at most $g+h^1(\xi^2\otimes
\sL^{-1})-1=2g+1<{\rm dim}(M)$ when $g>4$. We are done.
\end{proof}

Now we consider the case that $E$ is semi-stable of degree $0$ on
the generic fiber of $f: X\to C$. If $E$ is semi-stable on every
fiber of $f:X\to C$, then $E$ induces a non-trivial morphism
$$\varphi_E:C\to \Bbb P^1$$ (cf. \cite{FMW}) such that
$\varphi_E^*\sO_{\Bbb P^1}(1)=\Theta(E)=({\rm det}f_!E)^{-1}$, which
has degree $c_2(E)$ by Grothendieck-Riemann-Roch theorem. Thus
\ga{4.16}{\Delta(E)=4\cdot c_2(E)=4\cdot {\rm deg}(\varphi_E)\ge 8.}
If there is a $t_0\in C$ such that $E_{t_0}=E|_{X_{t_0}}$ is not
semi-stable on $X_{t_0}=f^{-1}(t_0)$, let $E_{t_0}\to \sO(\mu)\to 0$
be the quotient line bundle of minimal degree $\mu$ and $E'={\rm
kernel}\,(E\to \,_{X_{t_0}}\sO(\mu)\to 0\,)$, then we have

\begin{lem}\label{lem4.4} If $\Delta(E')=0$, then there is a semi-stable vector
bundle $V$ on $C$ and a line bundle $L$ of degree $0$ on $B$ such
that $$E'=f^*V\otimes \pi^*L.$$
\end{lem}

\begin{proof} By the definition, $\{E'_t=E'|_{\{t\}\times B}\}_{t\in C}$ and $\{E'_y=E'|_{C\times\{y\}}\}_{y\in
B}$ are families of semi-stable bundles of degree $0$. Apply Theorem
\ref{thm2.4} to $f: X\to C$ (resp. $\pi:X\to B$), then
$\Delta(E')=0$ implies that $\{E'_t\}_{t\in C}$ (resp.
$\{E'_y\}_{y\in B}$) are isomorphic each other. By tensor $E$ (thus
$E'$) with $\pi^*L^{-1}$ (where $L$ is a line bundle of degree $0$
on $B$), we can assume that ${\rm H}^0(E'_t)\neq 0$
($\forall\,\,t\in C$), which have dimension at most $2$ since $E'_t$
is semi-stable of degree $0$. If ${\rm H}^0(E'_t)$ has dimension
$2$, then $E'=f^*(f_*E')$ and we are done.

If ${\rm H}^0(E'_t)$ has dimension $1$, we will show a
contradiction. In fact, by the definition of $E'$, we have an exact
sequence \ga{4.17} {0\to E'\to E\to\, _{X_{t_0}}\sO(\mu)\to 0} where
$\sO(\mu)$ is a line bundle on $\{t_0\}\times B\cong B$ of degree
$\mu<0$. Then
$$V_1:=f_*E=f_*E'$$
is a line bundle on $C$. Since $\{E'_t\}_{t\in C}$ are isomorphic
each other and ${\rm H}^0(E'_t)$ has dimension $1$, we have the
exact sequence \ga{4.18} {0\to f^*V_1\to E'\to
f^*V_2\otimes\pi^*L_0\to 0} for a line bundle $V_2$ on $C$ and a
degree $0$ line bundle $L_0$ on $B$. If $L_0\neq \sO_B$, then
$R^if_*(f^*(V_2^{-1}\otimes V_1)\otimes L_0)=V_2^{-1}\otimes
V_1\otimes{\rm H}^i(L_0)=0$ ($i=0,\,1$), which implies ${\rm
H}^1(X,f^*(V_2^{-1}\otimes V_1)\otimes L_0)=0$ and \eqref{4.18} is
splitting. This is impossible since $E'_y$ is semi-stable of degree
$0$ and we can show that ${\rm deg}(V_1)={\rm deg}(f_*E)\le -1$ in
the following.

To prove that ${\rm deg}(f_*E)\le -1$, we consider the exact
sequence \ga{4.19} {0\to f^*f_*E\to E\to \sF\to 0} where
$\sF|_{f^{-1}(C\setminus\{t_0\})}$ is locally free of rank $1$ by
\eqref{4.18}. But $\sF$ is not locally free (otherwise
$c_2(E)=(c_1(E)-c_1(f^*f_*E))\cdot c_1(f^*f_*E)=0$) and for any
$y\in B$ the restriction of \eqref{4.19} to $X_y=\pi^{-1}(y)$
\ga{4.20}{0\to f_*E\to E_y\to \sF_y\to 0} is exact, which means that
$\sF$ is $B$-flat (cf. Lemma 2.1.4 of \cite{HL}). Thus, by Lemma
1.27 of \cite{Si}, there is a $y_0\in B$ such that $\sF_{y_0}$ has
torsion $\tau\neq 0$ since $\sF$ is not locally free. Then, since
$E_{y_0}$ is stable of degree $1$,
$${\rm deg}(\sF_{y_0})\ge 1+{\rm
deg}(\sF_{y_0}/\tau)>1+\mu(E_{y_0})=\frac{3}{2}$$ which implies
${\rm deg}(f_*E)\le -1$ by \eqref{4.20}.

We have shown that $L_0$ has to be $\sO_B$ and \eqref{4.18} has to
be \ga{4.21} {0\to f^*V_1\to E'\to f^*V_2\to 0} which is determined
by a class of ${\rm H}^1(X,f^*(V_1\otimes V_2^{-1}))$. However, note
$R^1f_*(f^*(V_1\otimes V_2^{-1}))=V_1\otimes V_2^{-1}\otimes {\rm
H}^1(\sO_B)=V_1\otimes V_2^{-1}$ and
$${\rm H}^0(C, V_1\otimes V_2^{-1})=0,$$
by using Leray spectral sequence, we have
$${\rm H}^1(C, V_1\otimes V_2^{-1})\cong {\rm H}^1(X,f^*(V_1\otimes
V_2^{-1})).$$ Hence there exists an extension $0\to V_1\to V\to
V_2\to 0$ on $C$ such that $E'\cong f^*V$, which contradicts the
assumption $${\rm dim}({\rm H}^0(\{t\}\times B, E'_t))=1.$$
\end{proof}

\begin{prop}\label{prop4.5} When $E$ is semi-stable of degree $0$ on the generic
fiber of $f:X\to C$, we have $\Delta(E)\ge 8$. If $C$ is not
hyper-elliptic and $\phi:B\to M$ passes through a $(1,1)$-stable
bundle, assume that $E$ defines an essential elliptic curve, then
$\Delta(E)\ge 12.$
\end{prop}

\begin{proof} If $E$ is semi-stable on each fiber $X_t=f^{-1}(t)$,
then $E$ induces a non-trivial morphism $\varphi_E:C\to \Bbb P^1$.
By \eqref{4.16}, \, $\Delta(E)\ge 8$.

If there is a $t_0\in C$ such that $E_{t_0}$ is not semi-stable,
then we have $$0\to E'\to E\to\, _{X_{t_0}}\sO(\mu)\to 0$$ where
$\sO(\mu)$ is a line bundle of degree $\mu$ on $B$. If
$\Delta(E')\neq 0$, then $\Delta(E')>0$ by Theorem \ref{thm2.4}. On
the other hand, $c_1(E')^2=0$ since $E'$ has degree $0$ on the
generic fiber of $X\to C$ and ${\rm Pic}(C\times B)={\rm
Pic}(C)\times {\rm Pic}(B)$. Thus $\Delta(E')=4\cdot c_2(E')\ge 4$,
and by Lemma \ref{lem2.6} $$\Delta(E)=\Delta(E')-4\mu\ge 8.$$

If $\Delta(E')=0$, by Lemma \ref{lem4.4}, we can assume that
$E'=f^*V$, then the sequence \eqref{4.17} induces a nontrivial
morphism $\varphi: B\to \Bbb P(V_{t_0}^{\vee})$ such that
$\sO(-\mu)=\varphi^*\sO_{\Bbb P(V_{t_0}^{\vee})}(1)$. Thus
$\Delta(E)=-4\mu\ge 8$.

Now we assume that $C$ is not hyper-elliptic and $\phi:B\to M$
passes through a $(1,1)$-stable bundle. If $E$ is semi-stable on
each fiber $X_t$, then $\Delta(E)=4\cdot {\rm deg}(\varphi_E)\ge 12$
by \eqref{4.16} since $C$ is not hyper-elliptic.

If there is $t_0\in C$ such that $E_{t_0}$ is not semi-stable, we
claim $\Delta(E')>0$ since $\phi:B\to M$ passes through a
$(1,1)$-stable bundle. Otherwise, $E'=f^*V$ where $V$ is a
$(1,0)$-stable by Lemma \ref{lem3.2}, then sequence \eqref{4.17}
implies that $\phi:B\to M$ factors through a Hecke curve, which
implies that $\phi:B\to M$ is not an essential elliptic curve. If
$E'$ is semi-stable on each fiber $X_t$, then $E'$ defines a
nontrivial morphism $\varphi_{E'}:C\to \Bbb P^1$ such that
$\varphi^*\sO_{\Bbb P^1}(1)=\Theta(E')=({\rm
det}f_!E')^{-1}=c_2(E')$. Thus $\Delta(E')=4\cdot{\rm
deg}(\varphi_{E'})\ge 12$ and $\Delta(E)=\Delta(E')-4\mu\ge 16.$

If there is $t_0'\in C$ such that $E'_{t_0}$ is not semi-stable,
then we have \ga{4.22} {0\to \sF\to E'\to\, _{X_{t'_0}}\sO(\mu')\to
0} where $\sF_y=\sF|_{C\times\{y\}}$ is stable of degre $-1$ for
generic $y\in B$ since $E'_y$ is stable of degree $0$. If
$\Delta(\sF)\neq 0$, it is clear that $\Delta(\sF)=4\cdot
c_2(\sF)\ge 4$ and $\Delta(E)=\Delta(\sF)-4\mu'-4\mu\ge 12.$ If
$\Delta(\sF)=0$, by Theorem \ref{thm2.4}, there is a stable vector
bundle $V'$ on $C$ such that $\sF_y\cong V'$ for all $y\in B$. Then
we can choose $\sF=f^*V'$, the sequence \eqref{4.22} induces a
nontrivial morphism $\varphi: B\to \Bbb P({V'}_{t'_0}^{\vee})$ such
that $\sO(-\mu')=\varphi^*\sO_{\Bbb P({V'}_{t'_0}^{\vee})}(1)$. Thus
$\Delta(E')=-4\mu'\ge 8$ and $\Delta(E)=\Delta(E')-4\mu\ge 12$.

\end{proof}

We have seen in Example \ref{3.6} the existence of essential
elliptic curves of degree $6(r,d)$ (which is $6$ in our case). Then
we have shown

\begin{thm}\label{thm4.6} Let $M=\sS\sU_C(2,\sL)$ be the moduli space of
rank two stable bundles on $C$ with a fixed determinant of degree
$1$. Then, when $C$ is generic, any essential elliptic curve
$\phi:B\to M$ has degree
$${\rm deg}\phi^*(-K_M)\ge 6$$
and ${\rm deg}\phi^*(-K_M)=6$ if and only if $\phi:B\to M$ factors
through $$\phi: B \xrightarrow\psi q^{-1}(\xi)=\Bbb P({\rm
H}^1(V_2^{\vee}\otimes V_1))\xrightarrow{\Phi_{\xi}}M$$ for some
$\xi=(V_1,V_2)$ such that $\psi^*\sO_{\Bbb P({\rm
H}^1(V_2^{\vee}\otimes V_1))}(1)$ has degree $3$.

\end{thm}

\begin{proof} By Proposition \ref{4.2}, Proposition \ref{4.3} and
Proposition 4.5, we have $\Delta(E)\ge 6$. The possible case
$\Delta(E)=6$ occurs only in Proposition \ref{4.2} when
$c_2(F_2)=0$. This implies that $E$ must satisfy
$$0\to
f^*V_1\otimes\pi^*\sO(\mu_1-\mu_2)\to E\to f^*V_2\to 0$$ which
defines $\psi:B\to\Bbb P({\rm H}^1(V_2^{\vee}\otimes V_1))$ such
that $\psi^*\sO_{\Bbb P({\rm H}^1(V_2^{\vee}\otimes V_1))}(1)$ has
degree $\mu_1-\mu_2$. Then $\Delta(E)=6$ and \eqref{4.1} imply
$\mu_1-\mu_2=3$.
\end{proof}

\begin{thm}\label{thm4.7} When $g>4$ and $C$ is generic, any
essential elliptic curve $\phi:B\to M=\sS\sU_C(2,\sL)$ that passes
through the generic point must have ${\rm deg}\phi^*(-K_M)\ge 12.$
\end{thm}

For $r>2$, let $M=\sS\sU_C(r,\sL)$ where $\sL$ is a line bundle of
degree $d$. What is the minimal degree of essential elliptic curves
on $M$ ? I expect the following conjecture to be true.

\begin{conj}\label{conj4.8} Let $\phi:B\to M=\sS\sU_C(r,\sL)^s$ be an
essential elliptic curve defined by a vector bundle $E$ on $C\times
M$. Then, when $C$ is a generic curve, we have $${\rm
deg}\phi^*(-K_M)\ge 6(r,d).$$ When $(r,d)\neq r$, then ${\rm
deg}\phi^*(-K_M)=6(r,d)$ if and only if it is an elliptic curve of
split type in Example \ref{3.6}. If $\phi:B\to M$ passes through the
generic point and $g>4$, then ${\rm deg}\phi^*(-K_M)\ge 6r$.
\end{conj}

\bibliographystyle{plain}

\renewcommand\refname{References}

\end{document}